\documentclass[12pt,dvipsnames]{article}

\usepackage{amsmath}
\usepackage{amsthm}
\usepackage{amsfonts}
\usepackage{stmaryrd}
\usepackage{amssymb}
\usepackage[all]{xy}
\usepackage{comment}
\usepackage{enumitem}

\usepackage{tikz-cd}
\usetikzlibrary{shapes}
\newcommand\C{\mathbb{C}}

\newcommand\N{\mathbb{N}}
\newcommand\R{\mathbb{R}}

\newcommand\Ha{\mathbb{H}}

\newcommand{\im}{\textup{im}}
\newcommand{\II}{\mathrm{I\!I}}

\theoremstyle{plain}
\newtheorem{theoremAlph}{Theorem}
\newtheorem{corollaryAlph}[theoremAlph]{Corollary}

\newtheorem{theorem}{Theorem}[section]
\newtheorem{lemma}[theorem]{Lemma}
\newtheorem{corollary}[theorem]{Corollary}
\newtheorem{proposition}[theorem]{Proposition}

\theoremstyle{definition}
\newtheorem{definition}[theorem]{Definition}

\theoremstyle{remark}
\newtheorem{remark}[theorem]{Remark}

\theoremstyle{definition}

\usepackage{color}

\usepackage{xcolor}

\usepackage{soul}
\setstcolor{ForestGreen}

\usepackage{hyperref}
\hypersetup{
	colorlinks=true,
	linkcolor=blue,
	citecolor=blue,
	urlcolor=blue
}

\advance \topmargin by -\headheight
\advance \topmargin by -\headsep
     
\evensidemargin \oddsidemargin
\makeatletter
\newcommand{\subjclass}[2][1991]{%
	\let\@oldtitle\@title%
	\gdef\@title{\@oldtitle\footnotetext{#1 \emph{Mathematics subject classification.} #2}}%
}

\begin{document}

\renewcommand{\labelenumi}{(\roman{enumi})}
\renewcommand{\labelenumii}{(\alph{enumii})}

\title{Positive intermediate Ricci curvature on connected sums}
\author{Philipp Reiser\thanks{The author acknowledges funding by the SNSF-Project 200020E\textunderscore 193062 and the DFG-Priority programme SPP 2026.} \ and David J. Wraith}

\maketitle

\begin{abstract}
We consider the problem of performing connected sums in the context of positive $k^{th}$ intermediate Ricci curvature. We show that such connected sums are possible if the manifolds involved possess `$k$-core metrics' for some $k$. Here, a $k$-core metric is a generalization of the notion of core metric introduced by Burdick for positive Ricci curvature. Further, we show that connected sums of linear sphere bundles over bases admitting such metrics admit positive $k^{th}$ intermediate Ricci curvature for $k$ in a particular range. This follows from a plumbing result we establish, which generalizes other recent plumbing results in the literature and is possibly of independent interest. As an example of a manifold admitting a $k$-core metric, we prove that $\Ha P^n$ admits a $(4n-3)$-core metric and that $\mathbb{O}P^2$ admits a $9$-core metric, and we show that in both cases these are optimal.

\end{abstract}

\section{Introduction}\label{intro}

Given two (or more) closed manifolds with the same dimension, the operation of connected sum is perhaps the most basic topological operation one can perform that again yields a closed manifold with the same dimension. If one wants to understand the interplay between the topology of manifolds and some aspect of geometry, investigating how that geometry behaves under connected sums is thus a fundamental task.

In Riemannian geometry, it is natural to ask what kind of curvature bounds can be preserved under connected sums. In general, this turns out to be a very hard question.

Given two manifolds, both with dimension $d \ge 3$ and Riemannian metrics of positive scalar curvature, Gromov--Lawson \cite{GL} famously showed that the connected sum also admits a metric of positive scalar curvature. Moreover, this metric can be chosen so as to agree with the original metrics outside a neighbourhood of the connected sum.

At the other extreme, it is so difficult to find metrics of positive sectional curvature that the question of behaviour under connected sums is unreasonable. In general, it follows from Gromov's Betti number bound \cite{Gr} that arbitrary connected sums cannot preserve positive sectional curvature. Cheeger \cite{Ch} showed that the connected sum of a pair of compact rank-one symmetric spaces admits a metric with non-negative sectional curvature. However it is unknown, for example, whether the connected sum of three such spaces admits non-negative sectional curvature. Indeed the Bott Conjecture on the rational ellipticity of simply-connected closed manifolds admitting non-negative sectional curvature implies that such metrics should not exist.

In between the scalar and sectional curvatures lies the Ricci curvature. The question of whether the connected sum of two manifolds with positive Ricci curvature also supports a metric of positive Ricci curvature turns out to be intriguing. By the theorem of Bonnet--Myers the connected sum of two closed, non-simply-connected Ricci-positive manifolds cannot admit a metric of positive Ricci curvature. However, if at least one of the manifolds is simply-connected, the question is open. This problem was systematically studied by Burdick \cite{Bu1,Bu2,Bu3}, who, based on earlier work by Perelman \cite{Pe}, introduced the notion of \emph{core metrics} and showed that the connected sum of manifolds with core metrics admits a metric of positive Ricci curvature.

%We consider the question of whether a given curvature condition can be preserved under connected sums. While this is always possible for positive scalar curvature as shown by Gromov--Lawson \cite{GL}, Gromov's Betti number bound \cite{Gr} implies that the same cannot be true for positive sectional curvature. For positive Ricci curvature, the situation is more delicate: By the theorem of Bonnet--Myers the connected sum of two closed, non-simply-connected Ricci-positive manifolds cannot admit a metric of positive Ricci curvature. However, if at least one of the manifolds is simply-connected, the question is open. This problem was systematically studied by Burdick \cite{Bu1,Bu2,Bu3}, who, based on earlier work by Perelman \cite{Pe}, introduced the notion of \emph{core metrics} and showed that the connected sum of manifolds with core metrics admits a metric of positive Ricci curvature.

In this article, we consider a natural family of positive curvature conditions which interpolate between positive Ricci curvature and positive sectional curvature:
%extend Burdick's techniques to a family of stricter curvature conditions:
\begin{definition}
	A Riemannian manifold $(M^n,g)$ has \emph{positive $k$-th intermediate Ricci curvature} for some $k\in\{1,\dots,n-1\}$, denoted $Ric_k>0$, if for every unit tangent vector $v\in TM$ and any orthonormal $k$-frame $(e^i)$ in $v^\perp$ the sum $\sum_{i=1}^{k}K(v,e^i)$ is positive, where $K$ denotes the sectional curvature.
\end{definition}
Note that for $k=1$ and $k=n-1$, we recover the conditions of positive sectional curvature and positive Ricci curvature, respectively. Although intermediate curvatures have appeared in the literature for several decades, in recent times there has been a dramatic increase in interest in these curvatures. For an up-to-date list of papers which feature intermediate curvatures, see \cite{Mo}.
\smallskip

The main goal of this paper is to establish conditions under which connected sums admit metrics with $Ric_k>0$.

Our first main result, Theorem \ref{k-core_conn_sum}, provides a generalization of Burdick's results to intermediate Ricci curvatures. This theorem requires a generalization of Burdick's notion of core metric, and we illustrate this new notion with reference to projective spaces (Theorem 
\ref{k-core_examples}).

The plumbing of disc bundles has proved to be a very important topological construction in the realm of positive Ricci curvature. See for example \cite{CW}. We prove a plumbing result for $Ric_k>0$ (Theorem \ref{plumbing}), and then illustrate this by providing examples of connected sums between linear sphere bundles which admit metrics with positive intermediate Ricci curvatures (Corollary \ref{sphere-bundles_conn-sum}).
\smallskip

In order to give a precise statement of the results, we must begin by defining our generalization of Burdick's core metrics:

	\begin{definition}\label{k-core}
		Let $M$ be an $n$-dimensional manifold and let $k\in\{1,\dots,n-1\}$. A Riemannian metric $g$ on $M$ is called a \emph{$k$-core metric} if $g$ has $Ric_k>0$ and if there exists an embedding $\varphi\colon D^n\hookrightarrow M$ such that
		\begin{enumerate}
			\item The induced metric $g|_{\varphi(S^{n-1})}$ is the round metric of radius one, and
			\item $\II_{\varphi(S^{n-1})}$ is positive semi-definite with respect to the outward normal of $S^{n-1}\subseteq D^n$.
		\end{enumerate}
	\end{definition}
	Note that for $k=n-1$ we recover the original definition given in \cite{Bu2} except for the fact that the second fundamental form is required to be strictly positive in \cite{Bu2}. However, a core metric in the sense of Definition \ref{k-core} can always be deformed into a core metric in the sense of \cite{Bu2}, see e.g.\ \cite[Proposition 1.2.11]{Bu1}.
	
	In \cite[Theorem B]{Bu2}, it is shown that connected sums of manifolds with $(n-1)$-core metrics support positive Ricci curvature. We can now generalize this as follows.
	
	\begin{theoremAlph}\label{k-core_conn_sum}
		Let $M_1,\dots,M_\ell$ be $n$-dimensional manifolds that admit $k$-core metrics, where $k\geq 2$. Then $M_1\#\dots\# M_\ell$ admits a metric with $Ric_k>0$.
	\end{theoremAlph}

	The main ingredients in the proof of Theorem  \ref{k-core_conn_sum} are the gluing theorem for positive intermediate Ricci curvature established in \cite{RW3}, together with the construction of a metric with $Ric_2>0$ on $S^n\setminus\sqcup_{\ell}(D^n)^\circ$, (which is called the \emph{docking station} in \cite{Bu2}),  whose second fundamental form on the boundary can be made arbitrarily small, see Theorem \ref{docking_station}.

	\begin{remark}
		Since the metric on the docking station is invariant under the action of $O(n-1)O(2)\subseteq O(n+1)$, we can take quotients by finite subgroups of $O(n-1)O(2)$ that act freely as in \cite[Corollary 4.7]{Bu2}. In this way we obtain in the situation of Theorem \ref{k-core_conn_sum} that $\R P^n\# M_1\#\dots\# M_\ell$ and $L\# M_1\#\dots\# M_\ell$ admits a metric of $Ric_k>0$, where $L$ is any $n$-dimensional lens space (and $n$ is assumed to be odd in this case). Note that by \cite[Lemma 1.2.9]{Bu1}, lens spaces and real projective spaces are the only additional summands we can obtain in this way.
	\end{remark}

	Concerning the existence of $k$-core metrics, by a result of Wu \cite{Wu}, the boundary condition (ii) in Definition \ref{k-core} imposes the following topological obstruction.
	
	\begin{proposition}\label{k-core_obstr}
		Let $M$ be a closed $n$-dimensional manifold that admits a $k$-core metric. Then $M$ is $(n-k)$-connected. In particular, if $k\leq \lfloor\frac{n+1}{2}\rfloor$, then $M$ is a homotopy sphere.
	\end{proposition}
	
	We immediately obtain the following restrictions in low dimensions: Every closed 3-manifold with a $k$-core metric is diffeomorphic to the standard sphere and the same holds in dimension 5 when $k\leq 3$. In dimension $4$ every closed manifold with a $k$-core metric is homeomorphic to the standard sphere when $k\leq 2$.

		On the other hand, it is easy to see that the round metric on $S^n$ is a $1$-core metric. Further, by \cite{Bu2}, complex and quaternionic projective spaces and the Cayley plane of dimension $n$ admit $(n-1)$-core metrics, where $n$ denotes the real dimension of the corresponding manifold. By Proposition \ref{k-core_obstr}, this is optimal for complex projective spaces. For quaternionic projective spaces and the Cayley plane we obtain the following improvement, which again is optimal by Proposition \ref{k-core_obstr}.
		
		\begin{theoremAlph}\label{k-core_examples}
			$\Ha P^{n}$ admits a $(4n-3)$-core metric and $\mathbb{O} P^2$ admits a $9$-core metric.
		\end{theoremAlph}
	
		In \cite{RW1} it was shown that a Betti number bound as in the case of non-negative sectional curvature \cite{Gr} cannot hold for $Ric_k>0$ for all $k\geq\lfloor \frac{n}{2}\rfloor+2$, where $n$ denotes the dimension. By considering connected sums of copies of $\Ha P^2$ and $\mathbb{O}P^2$ using Theorems \ref{k-core_conn_sum} and \ref{k-core_examples}, we can slightly improve this result as follows.
		\begin{corollaryAlph}
			For any $\ell\in\N$ the manifold $\#_\ell \Ha P^2$ admits a metric of $Ric_5>0$ and the manifold $\#_\ell \mathbb{O} P^2$ admits a metric of $Ric_9>0$. In particular, Gromov's Betti number bound does not hold in dimension $8$ for $Ric_5>0$ and in dimension $16$ for $Ric_9>0$.
		\end{corollaryAlph}

		By using manifolds with $k$-core metrics as base manifolds of fibre bundles, we can also consider plumbings as in the following theorem, which generalizes results for positive Ricci curvature in \cite{Re1} and \cite{Wr1}, and for positive intermediate Ricci curvature in \cite{RW2}.
		
		\begin{theoremAlph}\label{plumbing}
			Let $W$ be the manifold obtained by plumbing linear disc bundles with compact base manifolds according to a simply-connected graph. Suppose the following:
			\begin{enumerate}
				\item For a fixed bundle in this graph the base admits a metric with $Ric_{k_1}>0$ for some $k_1$. Denote the base dimension by $q+1$ and the fibre dimension by $p+1$,
				\item Every other bundle in this graph with base dimension $q+1$ admits a $k_1$-core metric,
				\item Every bundle with base dimension $p+1$ admits a $k_2$-core metric for some $k_2$.
			\end{enumerate}
			Then, if $p,q\geq 2$, the manifold $\partial W$ admits a metric of $Ric_k>0$ for all $k\geq \max\{p+2,p+k_1,q+2,q+k_2\}$.
		\end{theoremAlph}
	
		We can use plumbings as in Theorem \ref{plumbing} to construct connected sums of sphere bundles as follows.
		
		\begin{corollaryAlph}\label{sphere-bundles_conn-sum}
			Let $E_i\to B_i^q$, $1\leq i\leq \ell$, be linear $S^p$-bundles with compact base manifolds such that $B_1$ admits a metric of $Ric_k>0$ and each $B_i$, $2\leq i\leq \ell$, admits a $k$-core metric. Then the connected sum $E_1\#\dots\# E_\ell$ admits a metric of $Ric_k>0$ for all $k\geq \max\{p+2,p+k,q+1\}$.
		\end{corollaryAlph}

This paper is laid out as follows. In Section 2 we prove a generalization of the main technical result in \cite{RW1}. The aim is to establish criteria which identify when metrics (of the type under consideration in this paper) have $Ric_k>0.$ In Section 3 we prove that the neck construction from \cite{Pe} actually gives a metric with $Ric_2>0$, and we use this to prove Theorem \ref{k-core_conn_sum}. The remaining results (Theorems \ref{k-core_examples}, \ref{plumbing}, and Corollary \ref{sphere-bundles_conn-sum}) are then established in Section 4.

	\section{Preliminaries}

	Let $(M^n,g)$ be a Riemannian manifold. To characterize the condition $Ric_k>0$ we consider the curvature operator $\mathcal{R}\colon \Lambda^2 TM\to\Lambda^2TM$ defined by
	\[ g(\mathcal{R}(v_1\wedge v_2), v_3\wedge v_4)=g(R(v_1,v_2)v_4,v_3), \]
	where $\Lambda^2 TM$ is equipped with the Riemannian metric which is the natural extension of $g$ to $\Lambda^2 TM$, i.e.\
	\[ g(v_1\wedge v_2,v_3\wedge v_4)=g(v_1,v_3)g(v_2,v_4)-g(v_1,v_4)g(v_2,v_3). \]
	
	We recall the following definitions of \cite{RW1}: For an inner product space $V$ the set $\{v_0\wedge v_1,\dots,v_0\wedge v_k \}\subseteq\Lambda^2 V$, where $(v_0,\dots,v_k)$ is an orthonormal $(k+1)$-frame in $V$, is called a \emph{$k$-chain} with \emph{base} $v_0$. For a linear map $A\colon \Lambda^2 V\to\Lambda^2 V$ and a $k$-chain $\{v_0\wedge v_1,\dots,v_0\wedge v_k \}$ the sum
	\[ \sum_{i=1}^{k}\langle A(v_0\wedge v_i),v_0\wedge v_i\rangle \]
	is the \emph{value} of $A$ on this $k$-chain. Note that $(M,g)$ has $Ric_k>0$ if and only if at every point in $M$ the value of $\mathcal{R}$ on every $k$-chain is positive.
	
	In \cite{RW1} we considered the condition $Ric_k>0$ for doubly warped product metrics. In this case each tangent space splits orthogonally into a direct sum $V_1\oplus V_2\oplus V_3$ such that each subspace $V_i\wedge V_j$ is an eigenspace for $\mathcal{R}$. Below we will be interested in the following more general situation.

	\begin{proposition}\label{algebra}
		Let $(V,\langle\cdot,\cdot\rangle)$ be a finite-dimensional inner product space of dimension $n$ and let $A\colon \Lambda^2 V\to\Lambda^2V$ be a linear self-adjoint map. Suppose that $V$ splits orthogonally as
		\[ V=V_1\oplus V_2\oplus V_3 \]
		so that $V_1$ and $V_2$ are one-dimensional and $A$ is given by
		\begin{align*}
			A(v_1\wedge v_2)=&\lambda_{12}v_1\wedge v_2,\\
			A(v_1\wedge w_1)=&\lambda_{13}v_1\wedge w_1+\tilde{\lambda}v_2\wedge w_1,\\
			A(v_2\wedge w_1)=&\lambda_{23}v_2\wedge w_1+\tilde{\lambda} v_1\wedge w_1,\\
			A(w_1\wedge w_2)=&\lambda_3 w_1\wedge w_2,
		\end{align*}
		for some $\lambda_{12},\lambda_{13},\lambda_{23},\tilde{\lambda},\lambda_3\in\R$, where $v_1$ and $v_2$ are unit vectors in $V_1$ and $V_2$, respectively, and $w_1,w_2\in V_3$. Then for $2\leq k\leq n-3$ the value of $A$ on every $k$-chain is positive if and only if the following inequalities are satisfied:
		\begin{enumerate}
			\item $\lambda_{12}+\frac{1}{2}(k-1)(\lambda_{13}+\lambda_{23}  )>0 $,
			\item $(\lambda_{12}+(k-1)\lambda_{13})(\lambda_{12}+(k-1)\lambda_{23})>(k-1)^2\tilde{\lambda}^2$,
			
			\item $\lambda_{13}\lambda_{23}>\tilde{\lambda}^2$,
			\item $\lambda_{13},\lambda_{23},\lambda_{3}>0$.
		\end{enumerate}
		For $k=n-2,n-1$ these inequalities are still sufficient, but not necessary.
	\end{proposition}
	\begin{proof}
		First note that if $\tilde{\lambda}=0$, then the spaces $V_i\wedge V_j$ are eigenspaces for $A$, so we are in the situation of \cite[Proposition 2.3]{RW1}. Observe that (i)--(iv) in this case now become
		\begin{align*}
			\lambda_{12}+(k-1)\lambda_{13}&>0,\\
			\lambda_{12}+(k-1)\lambda_{23}&>0,\\
			\lambda_{13},\lambda_{23},\lambda_{3}&>0,
		\end{align*} 
		and these are precisely the inequalities appearing in \cite[Proposition 2.3]{RW1} for $k\leq n-3$, and for $k=n-2,n-1$ these inequalities are easily seen to be implied by those appearing in \cite[Proposition 2.3]{RW1}. %That these inequalities are necessary to conclude that the value of $A$ on every $k$-chain is positive can be seen by considering $k$-chains that consist entirely of eigenvectors.
		
		From now on we can therefore assume that $\tilde{\lambda}\neq0$.
		We modify the vectors $v_1$ and $v_2$ as follows. First, Let
		\[ \mu=\frac{\lambda_{13}-\lambda_{23}+\sqrt{(\lambda_{13}-\lambda_{23})^2 +4\tilde{\lambda}^2 }}{2\tilde{\lambda}}. \]
		and define
		\[ v_1'=\mu v_1+v_2,\quad v_2'=- v_1+\mu v_2. \]
		Let $V_1'$ and $V_2'$ be the subspaces generated by $v_1'$ and $v_2'$, respectively, and set $V_3'=V_3$. Then $V_1'$ and $V_2'$ are orthogonal and $V_1'\oplus V_2'=V_1\oplus V_2$. A calculation shows that the spaces $V_i'\wedge V_j'$ are eigenspaces for $A$ with eigenvectors $\lambda_{ij}'$ given by
		\begin{align*}
			\lambda_{12}'&=\lambda_{12},\\
			\lambda_{13}'&=\frac{1}{2}\left( \lambda_{13}+\lambda_{23}+\sqrt{(\lambda_{13}-\lambda_{23})^2+4\tilde{\lambda}^2 } \right)\\
			\lambda_{23}'&=\frac{1}{2}\left( \lambda_{13}+\lambda_{23}-\sqrt{(\lambda_{13}-\lambda_{23})^2+4\tilde{\lambda}^2 } \right)\\
			\lambda_{33}'&=\lambda_3.
		\end{align*}
		By \cite[Proposition 2.3]{RW1}, the value of $A$ on every $k$-chain is positive if and only if the sum of any $k$ non-diagonal elements in each row of the following $(n\times n)$ matrix is positive.
		\[
		\begin{pmatrix}
			0 & \lambda_{12}' & \lambda_{13}' & \cdots & \cdots & \lambda_{13}'\\
			\lambda_{12}' & 0 & \lambda_{23}' & \cdots & \cdots & \lambda_{23}'\\
			\lambda_{13}' & \lambda_{23}' & 0 & \lambda_{33}' & \cdots & \lambda_{33}' \\
			\vdots & \vdots & \lambda_{33}'& \ddots & \ddots & \vdots \\
			\vdots & \vdots & \vdots & \ddots & \ddots & \lambda_{33}' \\
			\lambda_{13}' & \lambda_{23}' & \lambda_{33}' & \cdots  & \lambda_{33}' & 0
		\end{pmatrix}
		\]
		For $2\leq k\leq n-3$ this is equivalent to the following system of inequalities.
		\begin{align*}
			\lambda_{12}'+(k-1)\lambda_{13}'&>0,\quad 
			&\lambda_{12}'+(k-1)\lambda_{23}'&>0,\\
			\lambda_{13}'+\lambda_{23}'+(k-2)\lambda_{33}'&>0,\quad
			&\lambda_{13}'+(k-1)\lambda_{33}'&>0,\\
			\lambda_{23}'+(k-1)\lambda_{33}'&>0,\quad
			&\lambda_{13}',\lambda_{23}',\lambda_{33}'&>0. 
		\end{align*}
		For $k=n-2,n-1$ these inequalities are still sufficient, but not all are necessary.
	
		Note that the inequalities $\lambda_{13}'+\lambda_{23}'+(k-2)\lambda_{33}'>0$, $\lambda_{13}'+(k-1)\lambda_{33}'>0$ and $\lambda_{23}'+(k-1)\lambda_{33}'>0$ are superfluous. Hence, we arrive at the following system of inequalities.
		\begin{align}
			\label{EQ1}\lambda_{12}+\frac{1}{2}(k-1)\left( \lambda_{13}+\lambda_{23}+\sqrt{(\lambda_{13}-\lambda_{23})^2+4\tilde{\lambda}^2 } \right)&>0,\\
			\label{EQ2}\lambda_{12}+\frac{1}{2}(k-1)\left( \lambda_{13}+\lambda_{23}-\sqrt{(\lambda_{13}-\lambda_{23})^2+4\tilde{\lambda}^2 } \right)&>0,\\
			\label{EQ6}\lambda_{13}+\lambda_{23}+\sqrt{(\lambda_{13}-\lambda_{23})^2+4\tilde{\lambda}^2 }&>0,\\
			\label{EQ7}\lambda_{13}+\lambda_{23}-\sqrt{(\lambda_{13}-\lambda_{23})^2+4\tilde{\lambda}^2 }&>0,\\
			\label{EQ8}\lambda_3&>0.
		\end{align}
		
		First note that \eqref{EQ2} implies \eqref{EQ1}, and \eqref{EQ7} implies \eqref{EQ6}. Hence, we are left with \eqref{EQ2}, \eqref{EQ7} and \eqref{EQ8}. Next, observe that \eqref{EQ7} implies $\lambda_{13}+\lambda_{23}>0$ and is therefore equivalent to
		\[ \lambda_{13}\lambda_{23}>\tilde{\lambda}^2. \]
		In particular, $\lambda_{13}\lambda_{23}>0$. This observation, together with $\lambda_{13}+\lambda_{23}>0$, is equivalent to
		\[\lambda_{13},\lambda_{23}>0.\]
		Hence, \eqref{EQ7} is equivalent to
		\[ \lambda_{13},\lambda_{23}>0,\quad \lambda_{13}\lambda_{23}>\tilde{\lambda}^2. \]
		Finally, \eqref{EQ2} is equivalent to $(i)$ and $(ii)$, since it is equivalent to
		\[ \lambda_{12}+\frac{1}{2}(k-1)(\lambda_{13}+\lambda_{23})>0\]
		and
		\[\left(\lambda_{12}+\frac{1}{2}(k-1)(\lambda_{13}+\lambda_{23}) \right)^2>\frac{(k-1)^2}{4}\left((\lambda_{13}-\lambda_{23})^2+4\tilde{\lambda}^2\right). \]
		A calculation now shows that the second inequality is equivalent to $(ii)$.
		
	\end{proof}
	
		\begin{remark}
			By adapting the arguments in the proof of Proposition \ref{algebra}, one can also obtain equivalent characterizations in the cases $k=1,n-2,n-1$. We omit this as it is not needed in this article.
		\end{remark}

\section{Perelman's neck construction}

In this section we prove the following result, which is the main ingredient in the proof of Theorem \ref{k-core_conn_sum}.

\begin{theorem}\label{docking_station}
	For any $\nu>0$ sufficiently small, $\ell\in\N$, $n\geq 3$ and all $k\geq 2$ there exists a metric of $Ric_k>0$ on $S^n\setminus \sqcup_{\ell}(D^n)^\circ$ such that the induced metric on each boundary component is the round metric of radius one and the principal curvatures are all given by $-\nu$.
\end{theorem}

The construction of the metric in Theorem \ref{docking_station} follows that of \cite{Pe} and consists of two parts: First, the \emph{ambient space}, which is a metric of positive sectional curvature on $S^n\setminus \sqcup_{\ell}(D^n)^\circ$, where the metric on each boundary component is a warped product metric whose 'waist' can be chosen arbitrarily small and with principal curvatures all at least -1. It is already established in \cite{Pe} that the metric has positive sectional curvature. Second, the \emph{neck}, which is a metric on $S^{n-1}\times[0,1]$ connecting the metrics on the boundary components of the ambient space to round metrics with constant and arbitrarily small second fundamental form. This metric on the neck is shown to have positive Ricci curvature in \cite{Pe} and we show below that it has in fact $Ric_2>0$:

\begin{proposition}\label{neck}
	Let $g$ be a metric on $S^n$, $n\geq 2$, of the form
	\[ g=dt^2+B^2(t)ds_{n-1}, \]
	where $t\in[0,\pi R]$, and we set $r=\max_t B(t)$. Assume that $g$ has sectional curvatures greater than $1$ and suppose that $r<R^2$. Let $\rho\in(r^{\frac{1}{2}},R)$. Then there exists a metric of $Ric_2>0$ on $S^n\times [0,1]$ such that
	\begin{enumerate}
		\item The induced metric on $S^n\times\{0\}$ is the round metric of radius $\frac{\rho}{\lambda}$ and satisfies $\II\equiv -\lambda$ for some $\lambda>0$,
		\item The induced metric on $S^n\times\{1\}$ is isometric to $g$ and satisfies $\II>1$.
	\end{enumerate}
\end{proposition}
The metric we will construct in the proof of Proposition \ref{neck} is of the form $dt^2+A(t,x)^2dx^2+B(t,x)^2ds_m^2$, where $dx^2$ denotes the standard metric on $S^1$. We first compute the curvatures of such a metric.

\begin{lemma}\label{neck_curv}
	Let $t_0<t_1$ and denote by $dt^2$ the standard metric on $[t_0,t_1]$, and by $dx^2$ the standard metric on $S^1$. Let $A,B\colon[t_0,t_1]\times S^1\to\R_{>0} $ be smooth positive functions and define the metric $g$ on $[t_0,t_1]\times S^1\times S^m$ by
	\[g=dt^2+A(t,x)^2dx^2+B(t,x)^2ds_m^2.  \]
	Let $v_1,v_2$ denote vectors tangent to $S^m$. Then the curvature tensor of $g$ is given by
	\begin{align*}
		\mathcal{R}(\partial_t\wedge\partial_x)=&-\frac{A_{tt}}{A}\partial_t\wedge\partial_x,\\
		\mathcal{R}(\partial_t\wedge v_1)=&-\frac{B_{tt}}{B}\partial_t\wedge v_1+\left(-\frac{B_{xt}}{A^2 B}+\frac{A_t B_x}{A^3 B}\right)\partial_x\wedge v_1,\\
		\mathcal{R}(\partial_x\wedge v_1)=& \left(-\frac{B_{xt}}{B}+\frac{A_t B_x}{AB}\right)\partial_t\wedge v_1+\left(-\frac{A_t B_t}{AB}-\frac{B_{xx}}{A^2 B}+\frac{A_x B_x}{A^3 B}\right)\partial_x\wedge v_1,\\
		\mathcal{R}(v_1\wedge v_2)=&\left(\frac{1-B_t^2}{B^2}-\frac{B_x^2}{A^2B^2}\right)v_1\wedge v_2.
	\end{align*}
\end{lemma}
\begin{proof}
	By using the Koszul formula one easily verifies that the Levi-Civita connection of $g$ is given by
	\begin{align*}
		\nabla_{\partial_t}\partial_t =& 0,\\
		\nabla_{\partial_t}\partial_x=&\nabla_{\partial_x}\partial_t=\frac{A_t}{A}\partial_x,\\
		\nabla_{\partial_t}v_1=&\nabla_{v_1}\partial_t = \frac{B_t}{B}v_1,\\
		\nabla_{\partial_x}\partial_x=& -AA_t\partial_t+\frac{A_x}{A}\partial_x,\\
		\nabla_{\partial_x}v_1=&\nabla_{v_1}\partial_x=\frac{B_x}{B}v_1,\\
		\nabla_{v_1} v_2=& -ds^2_m( v_1,v_2)\left(BB_t\partial_t+\frac{BB_x}{A^2}\partial_x\right)+\nabla^{S^m}_{v_1}v_2.
	\end{align*}
	From this one can now calculate the full curvature tensor.
\end{proof}

\begin{proof}[Proof of Proposition \ref{neck}]
	We use the same metric as constructed in \cite[Section 2]{Pe}. This metric is constructed as follows.
	
	We rewrite the metric $g$ as
	\[g=r^2\cos^2(x)ds_{n-1}^2+A^2(x)dx^2, \]
	$x\in[-\frac{\pi}{2},\frac{\pi}{2}]$, where $A$ satisfies $A(\pm\frac{\pi}{2})=r$, $A'(\pm\frac{\pi}{2})=0$. Then, since
	\[ \int_{-\frac{\pi}{2}}^{\frac{\pi}{2}}A(x)dx=\pi R, \]		
	there exists $x\in[-\frac{\pi}{2},\frac{\pi}{2}]$ with $A(x)\geq  R\, (>r),$ (note that $R<1$ by the theorem of Bonnet--Myers). Hence, we can rewrite $A$ as
	\[ A(x)=r(1-\eta(x)+\eta(x)a_{\infty}), \]
	where $\eta$ is a function satisfying $\max_x\eta(x)=1$, $\eta(\pm\frac{\pi}{2})=0$ and $\eta'(\pm\frac{\pi}{2})=0$, and $a_\infty\in\R$ with $a_\infty\geq\frac{R}{r}$.
	
	For $t_0<t_\infty$ we define the metric $g_{t_0,t_\infty}$ on $S^n\times[t_0,t_\infty]$ by
	\[ g_{t_0,t_\infty}=dt^2+A^2(t,x)dx^2+B^2(t,x)ds_{n-1}^2, \]
	where
	\[ B(t,x)=tb(t)\cos(x),\quad A(t,x)=tb(t)(1-\eta(x)+\eta(x)a(t)) \]
	and $a,b$ are functions satisfying $a(t_0)=1$, $a'(t_0)=0$, $b(t_0)=\rho$, $b'(t_0)=0$, $a(t_\infty)=a_\infty>1$ and $b(t_\infty)>r$. This metric will later be rescaled by $\frac{r}{t_\infty b(t_\infty)}$ to satisfy the required properties.
	
	Using Lemma \ref{neck_curv} we see that the curvature operator $\mathcal{R}_{g_{t_0,t_\infty}}$ of this metric has the form of the map $A$ in Proposition \ref{algebra} with
	\begin{align*}
		\lambda_{12}&=-\frac{A_{tt}}{A}=K(\partial_t\wedge\partial_x),\\
		\lambda_{13}&=-\frac{B_{tt}}{B}=K(\partial_t\wedge v),\\
		\tilde{\lambda}&=-\frac{B_{xt}}{A^2B}+\frac{A_t B_x}{A^3 B}=\frac{1}{n-1}Ric(\partial_t,\partial_x),\\
		\lambda_{23}&=-\frac{A_t B_t}{AB}-\frac{B_{xx}}{A^2 B}+\frac{A_xB_x}{A^3 B}=K(\partial_x\wedge v),\\
		\lambda_3&=\frac{1-B_t^2}{B^2}-\frac{B_x^2}{A^2B^2}=K(v_1\wedge v_2).
	\end{align*}
	Here $v,v_1,v_2$ are tangent to $S^n$.
	
	The functions $a$ and $b$ are now explicitly defined by
	\begin{align*}
		\frac{b'}{b} =& -\frac{\beta(t-t_0)}{2t_0^2\ln(2t_0)},&&t_0\leq t\leq 2t_0,\\
		\frac{b'}{b} =& -\frac{\beta \ln(2t_0)}{t\ln(t)^2},&&t\geq 2t_0,\\
		\frac{a'}{a} =& -\alpha\frac{b'}{b},&&t\geq t_0.
	\end{align*}
	
	The constants $\alpha$ and $\beta$ are defined by
	\begin{align*}
		\beta=& (1-\varepsilon)\frac{\ln(\rho)-\ln(r)}{1+\frac{1}{4\ln(2t_0)}},\\
		\alpha=& \frac{(1+\delta)}{\beta}\frac{\ln(a_\infty)}{1+\frac{1}{4\ln(2t_0)}} = \frac{(1+\delta)}{(1-\varepsilon)}\frac{\ln(a_\infty)}{\ln(\rho/r)}
	\end{align*}
	for some $\epsilon,\delta>0$ small. These values imply that $\int_{t_0}^\infty b'/b=(1-\epsilon)(\ln r-\ln \rho)$ and $\int^\infty_{t_0} a'/a=(1+\delta)\ln a_\infty.$ 
	
	Similarly as in \cite{Pe} we estimate $\alpha$ as follows: At a maximum point of $\eta$ we have $\eta(x)=1$ and $\eta'(x)\tan(x)=0$. Hence, the sectional curvatures of $g$ at this point satisfy (e.g.\ by applying Lemma \ref{neck_curv})
	\[K_g(\partial_x\wedge v)=\frac{1}{A(x)^2}\left(1-\frac{\sin(x)A'(x)}{\cos(x)A(x)}\right)=\frac{1}{r^2a_\infty^2}.  \]
	Since $K_g>1$, it follows that $a_\infty<1/r$. Thus,
	\[ \ln(a_\infty)<\ln\left( \frac{1}{r} \right)<\ln\left(\frac{1}{r}\frac{\rho^2}{r}\right)=2\ln\left(\frac{\rho}{r} \right). \]
	Hence, $\ln(a_\infty)/\ln(\rho/r)<2$.
	
	We also have $a_\infty\geq R/r>\rho/r$, so that $\ln(a_\infty)>\ln(\rho/r)$. Hence, for $\epsilon$ and $\delta$ sufficiently small, we have $\alpha\in(1,2)$.
	
	By choosing $\epsilon$ smaller if necessary, we can assume that $(\rho/r)^\epsilon g$ still has sectional curvatures at least 1.
	The following estimates are now established in \cite[Section 2]{Pe} for $t_0$ sufficiently large (see also \cite[Lemma 2.6, Corollary C.2.9 and Corollary C.3.3]{Bu1}).
	\begin{align*}
		\lambda_{23},\lambda_{3}&\geq \frac{c_1}{t^2},\\
		|\lambda_{12}|,|\lambda_{13}|,|\tilde{\lambda}|&\leq \frac{c_2\ln(t_0)}{t^2\ln(t)^2}
	\end{align*}
	for some $c_1,c_2>0$. To estimate $\lambda_{13}$ a calculation now shows that
	\[ \lambda_{13}=-\left(\frac{b''}{b}+\frac{2b'}{tb}\right)\geq \frac{c_3\ln(t_0)}{t^2\ln(t)^2} \]
	for some $c_3>0$.
	
	By Proposition \ref{algebra} we need to satisfy the following inequalities:
	\begin{align}
		\label{required_ineq1}\lambda_{12}+\frac{1}{2}(\lambda_{13}+\lambda_{23})&>0,\\
		\label{required_ineq2}(\lambda_{12}+\lambda_{13})(\lambda_{12}+\lambda_{23})&>\tilde{\lambda}^2,\\
		\label{required_ineq3}\lambda_{13}\lambda_{23}&>\tilde{\lambda}^2,\\
		\label{required_ineq4}\lambda_{13},\lambda_{23},\lambda_{3}&>0.
	\end{align}
	From the above estimates it follows directly that \eqref{required_ineq1}, \eqref{required_ineq3} and \eqref{required_ineq4} are satisfied for $t_0$ sufficiently large. For \eqref{required_ineq2} we show that
	\[ \lambda_{12}+\lambda_{13}>\frac{c_4\ln(t_0)}{t^2\ln(t)^2} \]
	for some $c_4>0$, from which \eqref{required_ineq2} follows. We calculate
	\begin{align*}
		\lambda_{12}+\lambda_{13}&=\left( \frac{\alpha\eta a}{1-\eta+\eta a}-2 \right)\left( \left(\frac{b'}{b}\right)'+\frac{2b'}{tb} \right)-2\left(\frac{b'}{b}\right)^2-\frac{\eta a}{1-\eta+\eta a}\left( 2\frac{a'b'}{ab}+\left( \frac{a'}{a} \right)^2 \right).
	\end{align*}
	Similarly as in \cite[end of p.\ 161]{Pe} we see that, since $\alpha<2$ and $\eta\leq 1$, the first factor in the first summand is negative and uniformly bounded from above. Hence, the first summand is bounded from below by $\frac{c_5\ln(t_0)}{t^2\ln(t)^2}$ for some $c_5>0$ and the absolute value of the remaining terms is bounded from above by $\frac{c_6\ln(t_0)}{t^2\ln(t)^4}$ for some $c_6>0$. It follows that the required estimate holds for $t_0$ sufficiently large. Thus, the metric has $Ric_2>0$ for $t_0$ sufficiently large.
	
	Note that $\delta$ can still be chosen freely (which then determines $t_\infty$ via $a(t_\infty)=a_\infty$). This is now done as in \cite{Pe} to ensure that the required conditions on the principal curvatures are satisfied.
\end{proof}

We can now give the proof of Theorems \ref{docking_station} and \ref{k-core_conn_sum}. For this, we recall the following gluing theorem which was established in \cite{RW3}. 
\begin{theorem}[{\cite[Theorem A]{RW3}}]\label{gluing}
	Let $(M_1^n, h_1)$ and $(M_2^n, h_2)$ be Riemannian manifolds of $Ric_k>0$ for some $1\le k\le n-1$ with compact boundaries, and let $\phi:(\partial M_1, h_1|_{\partial M_1})\rightarrow (\partial M_2, h_2|_{\partial M_2} )$ be an isometry. If the sum of second fundamental forms $\II_{\partial M_1}+\phi^*\II_{\partial M_2}$ is positive semi-definite, then $M_1\cup_{\phi} M_2$ admits a smooth metric of $Ric_k>0$ which coincides with the 
	$C^{0}$-metric $h=h_1\cup_{\phi} h_2$ outside an arbitrarily small neighbourhood of the gluing area.
\end{theorem}

We will also need the following result of Perelman:
\begin{proposition}[{\cite[Section 3]{Pe}}]\label{ambient_space}
	For every $n\geq 3$, $\ell\geq 0$, $R_0\in(0,1)$ and $r>0$ sufficiently small there exists a metric $g$ on $S^n\setminus\sqcup_\ell (D^n)^\circ$ such that
	\begin{enumerate}
		\item $g$ has positive sectional curvature,
		\item The induced metric on each boundary component is of the form $dt^2+B(t)^2ds_{n-2}^2$ with $t\in[0,\pi \cos(r)]$ and $\max_t B(t)=\cos(r) R_0\frac{\sin(r+r^4/4)}{\sin(r)}$, and has sectional curvature at least 1, and
		\item The principal curvatures at each boundary are all at least -1.
	\end{enumerate}
\end{proposition}

\begin{proof}[Proof of Theorem \ref{docking_station}]
	We equip $S^n\setminus\sqcup_{\ell}(D^n)^\circ$ with the metric provided by Proposition \ref{ambient_space}, where $R_0$ is so small that $R_0<\nu^2$, and $r$ is so small so that $\cos(r)>\nu$ and $\cos(r)R_0\frac{\sin(r+r^4)}{\sin(r)}<\nu^2$. Hence, using Theorem \ref{gluing}, we can glue a copy of the neck obtained in Proposition \ref{neck} to each of the $\ell$ boundary components of $S^n\setminus\sqcup_{\ell}(D^n)^\circ$ to obtain a metric of $Ric_2>0$ on the resulting manifold. Note that $\cos(r)$ in Proposition \ref{ambient_space} corresponds to $R$ in Proposition \ref{neck} and $\cos(r)R_0\frac{\sin(r+r^4)}{\sin(r)}$ in Proposition \ref{ambient_space} corresponds to $r$ in Proposition \ref{neck}, and we choose $\rho=\nu$. Finally, we rescale the metric by $\frac{\lambda}{\rho}$ so that the induced metric on each boundary component is the round metric of radius $1$ and the principal curvatures are all given by $-\rho=-\nu$.
\end{proof}

\begin{proof}[Proof of Theorem \ref{k-core_conn_sum}]

	The proof is essentially similar to the proof of \cite[Theorem B]{Bu2}. We denote by $\varphi_i\colon D^n\hookrightarrow M_i$ the embedding provided by Definition \ref{k-core}. We now slightly perturb the $k$-core metric on each $M_i\setminus \varphi_i(D^n)^\circ$, e.g.\ as in \cite[Proposition 1.2.11]{Bu1}, such that the second fundamental form is strictly positive. Let $\nu_0>0$ be the smallest principal curvature of all these metrics. Thus, by Theorem \ref{gluing}, we can glue each $M_i\setminus \varphi_i(D^n)^\circ$ to $S^n\setminus\sqcup_{\ell}(D^n)^\circ$ with the metric provided by Theorem \ref{docking_station} by choosing $\nu<\nu_0$. Hence, we obtain a metric of $Ric_k>0$ on the connected sum $M_1\#\dots\# M_\ell$.
	
\end{proof}

\section{$k$-core metrics}

In this section we consider $k$-core metrics. We begin by restating Proposition \ref{k-core_obstr}. 

\begin{proposition}\label{k-core_obstr2}
	Let $M$ be a closed $n$-dimensional manifold that admits a $k$-core metric. Then $M$ is $(n-k)$-connected. In particular, if $k\leq \lfloor\frac{n+1}{2}\rfloor$, then $M$ is a homotopy sphere.
\end{proposition}

\begin{proof}
	Since the boundary of $M\setminus\varphi(D^n)^\circ$ has positive semi-definite second fundamental form, it follows from \cite[Theorem 1]{Wu} that $M\setminus\varphi(D^n)^\circ$ is obtained from $\varphi(S^{n-1})$ by attaching cells of dimension at least $n-k+1$. By viewing $\varphi(D^n)$ as a $0$-cell, we obtain a CW structure for $M$ with no cells in dimensions between 1 and $n-k$. It follows that $M$ is $(n-k)$-connected.
	
	Now if $k\leq \lfloor\frac{n+1}{2}\rfloor$, we obtain by Poincaré duality that $M$ is a closed simply-connected manifold with non-trivial homology groups only in degrees $0$ and $n$. Hence, $M$ is a homotopy sphere.
\end{proof}

We will now consider examples of manifolds with $k$-core metrics and applications to plumbing.
	
\subsection{Projective spaces}
	
To prove Theorem \ref{k-core_examples}, we will adapt the construction in \cite{Ch}, where a metric of non-negative sectional curvature and round totally geodesic boundary is constructed on $\C P^n$, $\Ha P^n$ and $\mathbb{O} P^2$ with a disc removed. We will follow \cite[Sections 3 and 4]{BM} and also include the arguments for $\C P^n$ as they are entirely similar.

The key observation is that, by considering cohomogeneity-one actions on these spaces, they can all be written as a disc bundle $G\times_K D\to G/H$, where $H\subseteq K\subseteq G$ are compact Lie groups. Here $K$ acts by isometries on a Euclidean vector space $V$ with principal isotropy group $H$ via a representation $\rho\colon K\to O(V)$ , and $D\subseteq V$ is the unit disc. The corresponding groups are given as follows, see \cite[Section 6.3]{AB}, \cite[Section 4.1]{BM} and \cite[Example 1]{Iw}.

\renewcommand{\arraystretch}{1.5}
\begin{table}[h!]
	\centering
	\begin{tabular}{|l|l|l|l|}
		\hline
		& $G$ & $K$ & $H$\\\hline\hline
		$\C P^n\setminus{D^{2n}}^\circ$ & $U(n)$ & $U(n-1)U(1)$ & $U(n-1)$\\
		$\Ha P^n\setminus{D^{4n}}^\circ$ & $Sp(n)$ & $Sp(n-1)Sp(1)$ & $Sp(n-1)$\\
		$\mathbb{O}P^2\setminus{D^{16}}^\circ$ & $Spin(9)$ & $Spin(8)$ & $Spin(7)$\\\hline
	\end{tabular}
	\caption{Cohomogeneity one structure of projective spaces with a disc removed.}
	\label{table}
\end{table}

The representation $\rho$ is given by projection onto $U(1)$ (resp.\ $Sp(1)$) followed by inclusion into $O(2)$ (resp.\ $O(4)$) for $\C P^n$ (resp.\ $\Ha P^n$). For $\mathbb{O} P^2$ it is given by the covering map $Spin(8)\to SO(8)$.

%In the following we will write $\mathbb{K}$ for one of $\C$, $\Ha$ and $\mathbb{O}$ and assume that $G$, $K$ and $H$ are given accordingly.

We will construct a $k$-core metric on $G\times_K D$ by defining a $K$-invariant metric on $G\times D$, which then descends to $G\times_K D$ such that the projection $G\times D\to G\times_K D$ is a Riemannian submersion. On $G\times D$ we consider the metric
\[g=L+(dt^2+f(t)^2ds_m^2),\]
where $m=\dim(V)-1$, $L$ is a left-invariant metric on $G$ which is $Ad_K$-invariant and $f\colon [0,t_0]\to \R_{\geq 0}$ is a smooth function for some $t_0>0$ which is odd at $t=0$ with $f'(0)=1$ and $f(t)>0$ for $t\in(0,t_0]$.

Let $\mathfrak{g}=\mathfrak{k}\oplus\mathfrak{m}$ and $\mathfrak{k}=\mathfrak{h}\oplus\mathfrak{p}$ be $L$-orthogonal decompositions of the Lie algebras. For $X\in\mathfrak{k}$, $t\in[0,t_0]$ and $v\in S^{m}$ we denote by $X^*_{tv}$ the action field at $tv \in D$ defined by $X$, i.e.\ $X^*_{tv}=\frac{d}{ds}\rho\bigl(\exp_K(sX)\bigr)(tv)|_{s=0}$. Then the vertical and horizontal subspaces $\mathcal{V}_{(e,tv)}$ and $\mathcal{H}_{(e,tv)}$ of $T_{(e,tv)}(G\times D)$ with respect to $g$ are given for $t>0$ by
\begin{align}
	\notag\mathcal{V}_{(e,tv)}&=(\mathfrak{h}\oplus\{0\})\oplus \{(-X,X^*_{tv})\mid X\in \mathfrak{p} \},\\
	\mathcal{H}_{(e,tv)}&=(\mathfrak{m}\oplus\{0\})\oplus \{(f(t)^2BY,Y^*_{tv})\mid Y\in\mathfrak{p}\}\oplus\langle\partial_t\rangle  ,\label{EQ:VH}
\end{align}
where $B\colon \mathfrak{p}\to\mathfrak{p}$ is the $L$-symmetric and $Ad_H$-linear automorphism defined by $L(X,BY)=ds_m^2(X^*_{tv},Y^*_{tv})$, cf.\ \cite[Equation (3.1)]{BM}. For $t=0$ we have
\begin{align}
	\notag\mathcal{V}_{(e,0)}&=\mathfrak{k}\oplus\{0\},\\
	\mathcal{H}_{(e,0)}&=\mathfrak{m}\oplus T_0D.\label{EQ:VHt=0}
\end{align}

With this description we can now give the proof of Theorem \ref{k-core_examples}.

\begin{proof}[Proof of Theorem \ref{k-core_examples}]
	We equip $G$ with the left-$G$-invariant and right-$K$-invariant which induces the round metric on $G/H$ (note that this metric does not need to be normal homogeneous). For $\C P^n$ (resp.\ $\Ha P^n$) the restriction to $U(1)$ (resp.\ $Sp(1)$) is then biinvariant, hence it is the round metric of some radius. In particular, the map $B$ is a multiple of the identity map. For $\mathbb{O} P^2$, the action of $H$ on $\mathfrak{p}$ is irreducible, so $B$ is also a multiple of the identity map by Schur's Lemma.
	
	Hence, there exists $b\in\R$ so that $B=b\cdot \mathrm{Id}_{\mathfrak{p}}$. For $\epsilon>0$ we now define the metric $L_\epsilon$ on $G$ via
	\[L_\epsilon=(1+\epsilon)L|_{\mathfrak{k}}+L|_{\mathfrak{m}},  \]
	so $L_\epsilon$ is again left-$G$-invariant and right-$K$-invariant and the map $B_\epsilon$ is given by $\frac{1}{1+\epsilon}b\cdot\mathrm{Id}_{\mathfrak{p}}$. Then the metric
	\[g_\epsilon=L_\epsilon+(dt^2+f(t)^2ds_m^2) \]
	on $G\times D$ induces a metric $\check{g}_\epsilon$ on $G\times_K D$ such that the projection $G\times D\to G\times_K D$ is a Riemannian submersion. The metric induced on a slice $G\times_K S^m=G\times_K(K/H)\cong G/H$ for $t>0$ is then given by
	\[ \frac{f(t)^2\frac{b}{1+\epsilon}}{1+f(t)^2\frac{b}{1+\epsilon}}L_\epsilon|_{\mathfrak{p}}+L_\epsilon|_{\mathfrak{m}}=(1+\epsilon)\frac{f(t)^2\frac{b}{1+\epsilon}}{1+f(t)^2\frac{b}{1+\epsilon}}L|_{\mathfrak{p}}+L|_{\mathfrak{m}}, \]
	see e.g.\ \cite{Ch}, \cite{GZ}, \cite[Lemma 3.1]{BM}. In particular, if $f(t)=\sqrt{\frac{1+\epsilon}{b\epsilon}}$, then this metric coincides with the metric induced from $L$ on $G/H$, i.e.\ it is the round metric. Thus, we will assume from now on that for given $\epsilon$, the function $f$ (and the value of $t_0$) is chosen such that $f(t_0)=\sqrt{\frac{1+\epsilon}{b\epsilon}}$, so that the induced metric on the boundary of $G\times_K D$ is round. Moreover, we assume that $f'(t_0)\geq 0$, so the second fundamental form on the boundary is positive semi-definite.
	
	We will now analyse the curvatures of the metric $\check{g}_\epsilon$ on $G\times_K D$. We assume that $f''<0$, so the metric $h_f=dt^2+f(t)^2ds_m^2$ on $D$ has positive sectional curvature. We choose $\epsilon$ sufficiently small such that the metric induced on $G/H$ by $L_\epsilon$ also has positive sectional curvature. It then follows that the metric $\check{g}_\epsilon$ has non-negative sectional curvature, see \cite{Ch}, \cite[Lemma 4.1]{BM}. Thus, to determine the smallest value $k$ for which this metric has $Ric_k>0$, we only need to identify the 2-planes of vanishing curvature, i.e.\ for given $u\in T(G\times_K D)$ we need to determine the set
	\[Z_u=\{v\in u^\perp\setminus\{0\}\mid\sec^{\check{g}_\epsilon}(u\wedge v)=0 \}. \]
	
	Let $A$ denote the $A$-tensor of the Riemannian submersion $(G\times D,g_\epsilon)\to (G\times_K D,\check{g}_\epsilon)$ and decompose $A$ into $A=A^1+A^2$ according to the splitting \eqref{EQ:VH}, i.e.\ $A^1$ has image in $\mathfrak{h}\oplus\{0\}$ and $A^2$ has image in $\{(-X,X_{tv}^*)\mid X\in\mathfrak{p}  \}$. As in \cite[Proof of Lemma 4.1]{BM} we conclude that for horizontal vectors $u=(u_1,u_2), v=(v_1,v_2)$ in $T_{(e,tv)}(G\times D)$ with $t>0$, we have
	\[{A^1}_u v={A^{G/H}}_{u_1}v_1,  \]
	where $A^{G/H}$ is the $A$-tensor of the Riemannian submersion $G\to G/H$ (where we consider $G$ equipped with the metric $L_\epsilon$). It follows from the O'Neill formulas that
	\begin{align*}
		\check{g}_\epsilon( R^{\check{g}_\epsilon}(u,v)v,u)&=g_\epsilon(R^{g_\epsilon}(u,v)v,u )+3|A_u v|^2\\
		&=L_\epsilon(R^{L_\epsilon}(u_1,v_1)v_1,u_1 )+h_f(R^{h_f}(u_2,v_2)v_2,u_2 )+3|A^{G/H}_{u_1}v_1|^2+3|A^2_{u}v|^2\\
		&= L_\epsilon(R^{G/H}(u_1,v_1)v_1,u_1 )+h_f(R^{h_f}(u_2,v_2)v_2,u_2  )+3|A^2_u v|^2.
	\end{align*}
	Since both the metric on $G/H$ and the metric $h_f$ have strictly positive sectional curvature, this expression can only vanish if the pairs $(u_1,v_1)$ and $(u_2,v_2)$ are both linearly dependant. If we write, according to \eqref{EQ:VH},
	\begin{align*}
		u&=(u_1,u_2)=(X+f(t)^2B_\epsilon Y,Y_{tv}^*+\lambda\partial_t),\\
		v&=(v_1,v_2)=(X'+f(t)^2B_\epsilon Y',{Y'}_{tv}^*+\lambda'\partial_t),
	\end{align*}
	this is satisfied if and only if there exist $a_1,a_2\in\R$ such that
	\begin{align*}
		\quad\quad(X',Y')&=a_1(X,Y) &&\text{or} &(X,Y)&=(0,0), \quad\text{and}\quad\quad\\
		\quad\quad(Y',\lambda')&=a_2(Y,\lambda) &&\text{or} &(Y,\lambda)&=(0,0).\quad\quad
	\end{align*}
	If $Y\neq 0$, then $a_1=a_2$, hence $v=a_1 u$ and $Z_u$ is empty. Hence, we can assume that $Y=0$. Then, if $X,\lambda\neq 0$, we have $X'=a_1 X$, $\lambda'=a_2\lambda$ and $Y'=0$, hence $Z_u$ is contained in a 1-dimensional subspace. Thus, we are left with the cases $X=0$, $\lambda\neq 0$ and $X\neq 0$, $\lambda=0$. In the first case, we have $Y'=0$ and $\lambda'=a_2\lambda$, so $Z_u$ is contained in a $\dim(G/K)$-dimensional subspace. In the second case we have $Y'=0$ and $X'=a_1X$, so $Z_u$ is contained in a 1-dimensional subspace.
	
	Hence, we have shown that $Z_u$ is contained in a $\dim(G/K)$-dimensional subspace for all $u\in T_{e,tv}(G\times_K D)$ and $t>0$. By $G$-invariance of the metric $\check{g}_\epsilon$ this holds for all points $(g,tv)$ with $t>0$. Similar arguments using \eqref{EQ:VHt=0} show that this result extends to the case $t=0$. Thus, the metric $\check{g}_\epsilon$ has $Ric_{\dim(G/K)+1}>0$. For $\C P^n$ this gives a metric of $Ric_{2n-1}>0$, for $\Ha P^n$ a metric of $Ric_{4n-3}>0$ and for $\mathbb{O} P^2$ a metric of $Ric_9>0$.	
\end{proof}

\subsection{Generalized surgery and plumbing}

To prove Theorem \ref{plumbing} we need two additional results: A surgery result extending \cite[Theorem 3.2]{RW2} and \cite[Theorem A]{Re1} and a deformation result that ensures that we can satisfy the assumptions of the surgery theorem in our setting. For $\rho>0$ we denote by $S^p(\rho)$ the round sphere of radius $\rho$ and for $R,N>0$ we denote by $D_R^{q+1}(N)$ a geodesic ball of radius $R$ in $S^{q+1}(N)$.
\begin{theorem}\label{surgery}
	Suppose we have the following:
	\begin{enumerate}
		\item A Riemannian manifold $(M^{p+q+1},g_M)$ of $Ric_{k_1}>0$,
		\item An isometric embedding $\iota\colon S^{p}(\rho)\times D^{q+1}_R(N)\hookrightarrow (M,g_M)$, (which implies $k_1\geq \max(p+1,q+2)$),
		\item A linear $S^q$-bundle $E\xrightarrow{\pi}B^{p+1}$, where $B$ is compact and admits a $k_2$-core metric $g_B$.
	\end{enumerate}
	Then, if $p,q\geq 2$, for any $r>0$ sufficiently small, there exists a constant $\kappa=\kappa(p,q,R/N,g_B,r)$, such that if $\frac{\rho}{N}<\kappa$, then the manifold
	\[ M_{\iota,\pi}=M\setminus\im(\iota)^\circ\cup_{\partial}\pi^{-1}(B\setminus \varphi(D^{p+1})^\circ ) \]
	admits a metric of $Ric_k>0$ for all $k\geq \max(p+2,q+2,q+k_2)$. This metric coincides outside the gluing area with a submersion metric on $E$ with totally geodesic round fibres of radius $r$ and a scalar multiple of the metric $g_M$ on $M$.
\end{theorem}
\begin{proof}
	We equip $E$ with a submersion metric with totally geodesic and round fibres of radius $r$ according to a horizontal distribution which is integrable over $\varphi(D^{p+1})\subseteq B$. Then, for $r$ sufficiently small, this metric has $Ric_k>0$ for all $k\geq \max(p+2,q+k_2)$ by \cite[Corollary 3.1]{RW2}. Further, over $\varphi(D^{p+1})$, the metric is a product, in particular it is given over $\varphi(S^p)$ by $ds_p^2+r^2ds_q^2$. As noted below Definition \ref{k-core}, we can slightly deform the metric on $\pi^{-1}(B\setminus\varphi(D^{p+1})^\circ)$ so that the induced metric on the boundary remains unchanged and the second fundamental form on the boundary is strictly positive.
	
	Now, by \cite[Theorem C and Remark 4.2]{RW1} there exists a metric of $Ric_k>0$ on the manifold
	\[ M_\iota=M\setminus \im(\iota)^\circ\cup_\partial (D^{p+1}\times S^q) \]
	for all $k\geq \max(p,q)+2$ such that the metric near the centre of $D^{p+1}\times S^q$ is given by $D_{R'}^{p+1}(N')\times S^q(\rho')$, where the values of $R',N',\rho'$ can be chosen freely (provided $\frac{R'}{N'}<\frac{\pi}{2}$). We choose $\rho'=r$ and $R',N'$ so that the induced metric on $\partial D^{p+1}_{R'}(N')$ is $ds_p^2$ and the principal curvatures at the boundary are at least $-\epsilon$ for given $\epsilon>0$. (Note that they converge to $0$ as $\frac{R'}{N'}\to\frac{\pi}{2}$).
	
	It follows that $D_{R'}^{p+1}(N')\times S^q(\rho')$ and $\pi^{-1}(B\setminus \varphi(D^{p+1})^\circ)$ have isometric boundaries, and for $\epsilon$ sufficiently small the principal curvatures of $\pi^{-1}(B\setminus \varphi(D^{p+1})^\circ)$ at the boundary are greater than those of $D_{R'}^{p+1}(N')\times S^q(\rho').$ Hence, by Theorem \ref{gluing}, we can replace $D^{p+1}\times S^q$ in $M_\iota$ by $\pi^{-1}(B\setminus \varphi(D^{p+1})^\circ)$ to construct $M_{\iota,\pi}$ while preserving $Ric_k>0$.
\end{proof}
To satisfy condition (ii) of Theorem \ref{surgery}, we need the following deformation result, which generalizes \cite[Theorem 1.10]{Wr2}.
\begin{lemma}\label{local_def_1-jet}
	Let $(M^n,g_0)$ be a Riemannian manifold of $Ric_k>0$ and let $N^p\subseteq M$ be a compact embedded submanifold. Let $g_1$ be a metric of $Ric_k>0$ defined in a tubular neighbourhood $U$ of $N$. If the $1$-jets of $g_0$ and $g_1$ on $N$ coincide, then there exists a metric $\tilde{g}$ of $Ric_k>0$ on $M$ that equals $g_0$ outside $U$ and equals $g_1$ on a (smaller) tubular neighbourhood of $N$.
\end{lemma}
\begin{proof}
	We consider for $t\in[0,1]$ the metric $g_t=(1-t)g_0+tg_1$ on $U$. Since the $1$-jets of $g_0$ and $g_1$ coincide on $N$ and since the sectional curvatures depend linearly on the second derivatives of the metric, we have $K_{g_t}=(1-t)K_{g_0}+tK_{g_1}$ on $N$. In particular, $g_t$ has $Ric_k>0$ on $N$ and by compactness this holds in a neighbourhood of $N$. By \cite[Theorem 1.2]{BH} the local deformation $g_t$ can now be extended to a global deformation of $g_0$, which leaves $g_0$ unchanged outside a neighbourhood of $N$ and coincides with the deformation $g_t$ on a (smaller) tubular neighbourhood of $N$.
\end{proof}
\begin{corollary}\label{local_def_curv1}
	Let $(M^n,g)$ be a Riemannian manifold of $Ric_k>0$ and let $p_1,\dots,p_\ell\in M$. Then the metric $g$ can be deformed into a metric of $Ric_k>0$ that has constant sectional curvature 1 in a neighbourhood of each $p_i$.
\end{corollary}
\begin{proof}
	We consider normal coordinates around each $p_i$, i.e.\ coordinates $(x_1,\dots,x_n)$ in which the metric is given by $g_{ab}=\delta_{ab}+O(r^2)$, where $r$ denotes the distance to $p_i$. In particular, the first derivatives $\partial_c g_{ab}$ all vanish at $p_i$. By considering normal coordinates at a point in the round sphere of radius 1, we obtain a second metric around each $p_i$ with the same property. Applying Lemma \ref{local_def_1-jet} now yields the required deformation.
\end{proof}

\begin{proof}[Proof of Theorem \ref{plumbing}]
	The proof of Theorem \ref{plumbing} follows the same lines as the proof of \cite[Theorem B]{Re1} by observing that $\partial W$ is obtained by iterated generalized surgeries as in Theorem \ref{surgery}. We simply replace \cite[Theorem A]{Re1} by Theorem \ref{surgery}, \cite[Proposition 2.2]{Re1} by \cite[Corollary 3.1]{RW2} and the deformation result used in the proof of \cite[Theorem B]{Re1} by Corollary \ref{local_def_curv1}.
\end{proof}

\begin{proof}[Proof of Corollary \ref{sphere-bundles_conn-sum}]
	Let $\overline{E}_i\to B_i$ denote the disc bundle corresponding to $E_i\to B_i$. We define $W$ as the manifold obtained by plumbing according to the following graph, where we denote by $\underline{D}^m_M$ the trivial disc bundle $M\times D^m\to M$ over a manifold $M$.

\[
\begin{tikzcd}[cells={nodes={ellipse, draw=black, anchor=center,minimum height=2em}}]
	& \underline{D}^{p+1}_{S^q} & & \underline{D}^{p+1}_{S^q} & &\\
	\overline{E}_1 \arrow[dash]{r} & \underline{D}^q_{S^{p+1}}\arrow[dash]{u}\arrow[dash]{r}& \overline{E}_2 \arrow[dash]{r} & \underline{D}^q_{S^{p+1}}\arrow[dash]{u}\arrow[dash]{r} & |[draw=none]|\dots \arrow[dash]{r} & \overline{E}_\ell
\end{tikzcd}
\]

	By \cite[Propositions 3.2 and 3.3]{Re2}, see also \cite[Proposition 2.6]{CW} and \cite[Section 5]{Bu2}, the manifold $\partial W$ is diffeomorphic to the connected sum $E_1\#\dots\# E_\ell$. By Theorem \ref{plumbing}, the manifold $\partial W$ admits a metric of $Ric_k>0$ for all $k\geq \max\{p+2,p+k,q+1,q\}=\max\{p+2,p+k,q+1\}$.

\end{proof}

\begin{remark}
	In \cite{Bu3} it is shown that the manifolds constructed in Theorem \ref{plumbing} and Corollary \ref{sphere-bundles_conn-sum} admit a core metric, provided each base manifold of the bundles involved admits a core metric. We conjecture that these manifolds in fact admit $k$-core metrics with $k$ as given in these results. However, this conjecture is open even in the simplest case of a linear sphere bundle over a manifold with a core metric (which can be viewed as a plumbing according to a graph with a single vertex).
\end{remark}

%%%%%%%%%%

\bigskip\bigskip

\noindent{\it Philipp Reiser, Department of Mathematics, University of Fribourg, Switzerland,
	Email: philipp.reiser@unifr.ch }\\

\bigskip

\noindent {\it David Wraith, 
	Department of Mathematics and Statistics, 
	National University of Ireland Maynooth, Maynooth, 
	County Kildare, 
	Ireland. 
	Email: david.wraith@mu.ie.}


\begin{thebibliography}{999}
	
\bibitem[AB]{AB} M. Alexandrino, R. Bettiol, \textit{Lie groups and geometric aspects of isometric actions}, Springer, Cham (2015), x+213 pp.
	
\bibitem[Bu1]{Bu1} B. L. Burdick, \textit{Metrics of Positive Ricci Curvature on Connected Sums: Projective Spaces, Products, and Plumbings}, ProQuest LLC, Ann Arbor, MI (2019). Thesis (Ph.D.)--University of Oregon.

\bibitem[Bu2]{Bu2} B. L. Burdick, \textit{Ricci-positive metrics on connected sums of projective spaces}, Differential Geom. Appl. {\bf 62} (2019), 212--233.

\bibitem[Bu3]{Bu3} B. L. Burdick, \textit{Metrics of positive Ricci curvature on the connected sums of products with arbitrarily many spheres}, Ann. Global Anal. Geom. {\bf 58(4)} (2020), 433--476.

\bibitem[BH]{BH} C. Bär, B. Hanke, \textit{Local flexibility for open partial differential relations}, Comm. Pure Appl. Math. \textbf{75} (2022), no.6, 1377--1415.

%\bibitem[BHSW]{BHSW} B. Botvinnik, B. Hanke, T. Schick, M. Walsh, \textit{Homotopy groups of the moduli space of metrics of positive scalar curvature}, Geom. Topol. \textbf{14} (2010), no. 4, 2047--2076.

\bibitem[BM]{BM} R. Bettiol, R. Mendes, \textit{Strongly nonnegative curvature}, Math. Ann. \textbf{368} (2017), no. 3--4, 971--986.

%\bibitem[BWW]{BWW} B. Botvinnik, M. Walsh, D. J. Wraith, \textit{Homotopy groups of the observer moduli space of Ricci positive metrics}, Geom. Topol. {\bf 23} (2019), 3003--3040.

%\bibitem[Cha]{Cha} I. Chavel, \textit{Riemannian geometry: a modern introduction}, Cambridge Tracts in Mathematics {\bf 108}, Cambridge University Press, (1995).

\bibitem[Ch]{Ch} J. Cheeger, \textit{Some examples of manifolds of nonnegative curvature}, J. Differential Geometry {\bf 8} (1973), 623--628.

\bibitem[CW]{CW} D. Crowley, D. J. Wraith, \textit{Positive Ricci curvature on highly connected manifolds}, J. Differential Geom. \textbf{106} (2017), no.2, 187--243.

\bibitem[Gr]{Gr} M. Gromov, {\em Curvature, diameter and Betti numbers}, Comm. Math. Helv. {\bf 56} (1981), 179--195.

%\bibitem[GL1]{GL1} M. Gromov, H. B. Lawson, \textit{Spin and scalar curvature in the presence of a fundamental group. I}, Ann. of Math. (2) {\bf 111} (1980), 209--230.

\bibitem[GL]{GL} M. Gromov, H. B. Lawson, \textit{The classification of simply connected manifolds of positive scalar curvature}, Ann. of Math. (2) {\bf 111} (1980), 423--434.

\bibitem[GZ]{GZ} K. Grove, W. Ziller, \textit{Curvature and symmetry of Milnor spheres}, Ann. of Math. (2) \textbf{152} (2000), no.1, 331--367.

%\bibitem[HSS]{HSS} B. Hanke, T. Schick and W. Steimle, \textit{The Space of Metrics of Positive Scalar Curvature}, Publ. Math. Inst. Hautes \'Etudes Sci. \textbf{120} (2014), 335--367. 

\bibitem[Iw]{Iw} K. Iwata, \textit{Compact transformation groups on rational cohomology Cayley projective planes}, Tohoku Math. J. (2) \textbf{33} (1981), no.4, 429--442.

%\bibitem[Ko]{Ko} N. N. Kosovskii, \textit{Gluing of Riemannian manifolds of curvature at least $\kappa$}, St. Petersbg. Math. J. 14, No. 3, 467--478 (2003); translation from Algebra Anal. 14, No. 3, 140--157 (2002).

%\bibitem[LM]{LM} H. B. Lawson, M.-L. Michelsohn, \textit{Embedding and surrounding with positive mean curvature}, Invent. Math. {\bf 77} (1984), 399--419.

%\bibitem[Mi]{Mi} P. Miao, \textit{Positive mass theorem on manifolds admitting corners along a hypersurface}, Adv. Theor. Math. Phys. {\bf 6} (2002), 1163--1182.

\bibitem[Mo]{Mo} L. Mouill\'e, \textit{Intermediate Ricci curvature database}: \url{https://sites.google.com/site/lgmouille/research/intermediate-ricci-curvature}

\bibitem[Re1]{Re1} P. Reiser, \textit{Generalized surgery on Riemannian manifolds of positive Ricci curvature}, Trans. Amer. Math. Soc. \textbf{376} (2023), no.5, 3397--3418.

\bibitem[Re2]{Re2} P. Reiser, \textit{Metrics of Positive Ricci Curvature on Simply-Connected Manifolds of Dimension $6k$}, arXiv:2210.15610 (2022).

\bibitem[RW1]{RW1} P. Reiser, D. J. Wraith, \textit{Intermediate Ricci curvatures and Gromov's Betti number bound}, J. Geom. Anal. {\bf 33} (2023), 364.

\bibitem[RW2]{RW2} P. Reiser, D. J. Wraith, \textit{Positive intermediate Ricci curvature on fibre bundles}, arXiv:2211.14610 (2022).

\bibitem[RW3]{RW3} P. Reiser, D. J. Wraith, \textit{A generalization of the Perelman gluing theorem and applications}, arXiv:2308.06996 (2023).

\bibitem[Pe]{Pe} G. Perelman, {\em Construction of manifolds of positive Ricci curvature with big volume and large Betti numbers}, Comparison Geometry, MSRI publications {\bf 30} (1997), 157--163.

%\bibitem[Sc]{S} A. Schlichting, {\em Gluing Riemannian manifolds with curvature operators at least $\kappa$}, arXiv:1210.2957 (2012).

%\bibitem[Sh1]{Sha1} J.-P. Sha, {\em $p$-convex Riemannian manifolds}, Invent. Math. {\bf 83} (1986), 437--447.

%\bibitem[Sh2]{Sha2} J.-P. Sha, {\em Handlebodies and $p$-convexity}, J. Differential Geom. {\bf 25} (1987), 353--361.

%\bibitem[TW]{TW} W. Tuschmann, D. Wraith \textit{Moduli Spaces of Riemannian Metrics,} Oberwolfach Seminars \textbf{46}, Birkhauser 2015.

%\bibitem[Wa]{Wa} C. T. C. Wall, {\em Geometrical connectivity. I}, J. London Math. Soc. (2) {\bf 3} (1971), 597--604.

%\bibitem[Wo]{Wo} J. Wolfson, {\em Manifolds with $k$-positive Ricci curvature}, Variational problems in differential geometry, 182--201, London Math. Soc. Lecture Note Ser., 394, Cambridge Univ. Press, Cambridge, 2012.

\bibitem[Wr1]{Wr1} D. J. Wraith, \textit{Exotic spheres with positive Ricci curvature}, J. Differential Geom. \textbf{45} (1997), no.3, 638--649.

\bibitem[Wr2]{Wr2} D. J. Wraith, \textit{Deforming Ricci positive metrics}, Tokyo J. Math. \textbf{25} (2002), no.1, 181--189.

\bibitem[Wu]{Wu} H. Wu, {\em Manifolds of partially positive curvature}, Indiana Univ. Math. J. {\bf 36(3)} (1987), 525--548.


\end{thebibliography}
\end{document}